\documentclass[12pt]{amsart}
\usepackage{amssymb}
\usepackage{amsmath}
\usepackage[active]{srcltx}
\usepackage{t1enc}
\usepackage[latin2]{inputenc}
\usepackage{verbatim}
\usepackage{amsmath,amsfonts,amssymb,amsthm}
\usepackage[mathcal]{eucal}
\usepackage{enumerate}
\usepackage[centertags]{amsmath}
\usepackage{graphics}

\newtheorem{theorem}{Theorem}

\newtheorem{lemma}{Lemma}

\begin{document}
\author{D. Baramidze$^{1}$,  L.-E. Persson$^{2,3}$, H. Singh$^{2}$, G. Tephnadze$^{1}$}

\title[logarithmic means \dots]{Some new results for subsequences of Nörlund logarithmic means of Walsh-Fourier series }
\address{D. Baramidze, The University of Georgia, School of science and technology, 77a Merab Kostava St, Tbilisi 0128, Georgia and Department of Computer Science and Computational Engineering, UiT - The Arctic University of Norway, P.O. Box 385, N-8505, Narvik, Norway.}
\email{Datobaramidze20@gmail.com }
\address{L.-E. Persson, UiT The Arctic University of Norway, P.O. Box 385, N-8505, Narvik, Norway and Department of Mathematics and Computer Science, Karlstad University, 65188 Karlstad, Sweden.}
\email{larserik6pers@gmail.com}
\address{H. Singh, Department of Computer Science and Computational Engineering, UiT - The Arctic University of Norway, P.O. Box 385, N-8505, Narvik, Norway.}
\email{harpal.singh@uit.no }
\address{G. Tephnadze, The University of Georgia, School of science and technology, 77a Merab Kostava St, Tbilisi 0128, Georgia}
\email{g.tephnadze@ug.edu.ge}
\date{}
\maketitle

$^{1}$ The University of Georgia, School of science and technology, 77a Merab Kostava St, Tbilisi 0128, Georgia.

$^{2}$ UiT The Arctic University of Norway, P.O. Box 385, N-8505, Narvik, Norway.

$^{3}$ Department of Mathematics and Computer Science, Karlstad University, 65188 Karlstad, Sweden.

\begin{abstract}
We prove that there exists a martingale $f\in H_p $ such that the
subsequence $\{L_{2^n}f \}$ of Nörlund logarithmic means with respect to the Walsh
system are not bounded in the Lebesgue space $weak-L_p $ for $0<p<1 $. Moreover, we  prove that for any $f\in L_p(G),$  $p\geq 1, $  $L_{2^n}f$ converge to $f$ at any Lebesgue point $x$. Some new related inequalities are derived.
\end{abstract}

\date{}



\section{Introduction}

The terminology and notations used in this introduction can be found in Section 2. 

It is well-known that Vilenkin systems do not form bases in the
space $L_{1}.$ Moreover, there is a function in the Hardy space $H_{1},$
such that the partial sums of $f$ \ are not bounded in the $L_{1}$-norm. Moreover, (see Tephnadze \cite{tep7}) there
exists a martingale $f\in H_{p}$ $\left( 0<p<1\right) ,$ such that
\begin{equation*}
\underset{n\in \mathbb{N}}{\sup }\left\Vert S_{2^n+1}f\right\Vert
_{weak-L_{p}}=\infty .
\end{equation*}%
The reason of the divergence of $S_{2^n+1}f$ \ is that when $0<p<1$ the
Fourier coefficients of $f\in H_{p}$ are not uniformly bounded (see
Tephnadze \cite{tep6}).

On the other hand, (for details see e.g. the books \ \cite{sws}
and \cite{We1}) the subsequence $\{S_{2^n}\}$ of partial
sums is bounded from the martingale Hardy space $H_{p}$ to the
space $H_{p},$ for all $p>0,$ that is the following inequality holds:
\begin{equation}\label{snjemala}
\left\Vert S_{2^n}f\right\Vert _{H_p}\leq c_{p}\left\Vert f\right\Vert _{H_p}, \ \ n\in \mathbb{N}, \ \ p>0.
\end{equation}

It is also well-known that (see \cite{sws})
\begin{equation}\label{s2nae}
S_{2^n}f(x)\to f(x),  \ \text{for all Lebesgue points of} \  f\in L_p(G), \ \text{where} \  p\geq 1.
\end{equation}

Weisz \cite{We2} considered the norm convergence of Fej\'er means of
Vilenkin-Fourier series and proved that the inequality
\begin{equation}
\left\Vert \sigma _{k}f\right\Vert _{p}\leq c_{p}\left\Vert f\right\Vert
_{H_{p}},\text{ \ \ \ }p>1/2\text{ \ \ \ and \ \ \ }f\in H_{p},  \label{f100}
\end{equation}
holds. Moreover, Goginava \cite{gog1} (see also \cite{PTT} and \cite{tep1}) proved that the assumption $p>1/2$ in (\ref%
{f100}) is essential. In particular, he showed that there exists a
martingale $f\in H_{1/2}$ such that
$
\sup_{n\in \mathbb{N}}\left\Vert \sigma _{n}f\right\Vert _{1/2}=+\infty .
$
However, Weisz \cite{We2} (see also \cite{pt}) proved that for every $ f\in H_p, $ there exists an absolute constant $ c_p, $ such that the following inequality holds:
\begin{equation} \label{sigmanjemala}
\left\Vert \sigma_{2^n}f\right\Vert _{H_p}\leq c_{p}\left\Vert f\right\Vert _{H_p}, \ \ n\in \mathbb{N}, \ \ p>0.
\end{equation}

M\'oricz and Siddiqi \cite{Mor} investigated the approximation properties of some special N\"orlund means of Walsh-Fourier series of $L_{p}$ functions in
norm. Similar results for the two-dimensional case  can be found in Nagy \cite{n,nagy}.  Approximation properties for general summability methods can be found in \cite{BNT,BN}. Fridli, Manchanda and Siddiqi \cite{FMS} improved and extended the results of M\'oricz and Siddiqi \cite{Mor} to martingale Hardy spaces.
The case when $\left\{ q_{k}=1/k:k\in \mathbb{N}\right\} $ was
excluded, since the methods are not applicable to N\"o%
rlund logarithmic means. In \cite{Ga2} G\'{a}t and Goginava proved some
convergence and divergence properties of the N\"orlund logarithmic means of
functions in the Lebesgue space $%
L_1.$ In particular, they proved that there exists an function in the space $ L_1, $ such that 
$
\sup_{n\in \mathbb{N}}\left\Vert L_{n}f\right\Vert _{1}=\infty .
$

In \cite{PTW} (see also \cite{BPT}) it was proved that there exists a martingale $f\in H_{p}, \ \ (0<p< 1)$ such that
$
\sup_{n\in \mathbb{N}}\left\| L_{n}f\right\| _{p}=\infty .
$

In \cite{PTW2} (see also \cite{TT1}) it was proved that there exists a martingale $f\in H_{1}$ such that 
\begin{equation} \label{main}
\sup\limits_{n\in \mathbb{N}}\left\Vert L_{n}f\right\Vert_{1}=\infty.
\end{equation}

However, Goginava \cite{gog2} proved that 
\begin{equation*}
\left\Vert L_{2^n}f\right\Vert _{1}\leq c\left\Vert f\right\Vert _{1}, \ \ f\in L_1, \ \ n\in \mathbb{N}.
\end{equation*}

From this result it immediately follows that for every $ f\in H_1, $ there exists an absolute constant $ c, $ such that the inequality
\begin{equation} \label{gogjemala}
\left\Vert L_{2^n}f\right\Vert _{1}\leq c\left\Vert f\right\Vert _{H_1}
\end{equation}
holds for all $n\in \mathbb{N}.$
Goginava \cite{gog2} also proved that for any $f\in L_1(G),$ 
\begin{equation*}
L_{2^n}f(x)\to f(x), \ \ \ \text{a.e., as} \ \ \  n\to\infty.
\end{equation*}

According to \eqref{snjemala}, \eqref{sigmanjemala} and \eqref{gogjemala}, the following question is quite natural: 

\textbf{Question 1.} Is the subsequence  $ \{L_{2^n} \}$ also bounded on the martingale Hardy spaces $H_p(G)$ when $ 0<p<1 ?$

In Theorem 2 of this paper we give a negative answer to this question. In particular, we further develop some methods considered in \cite{BPTW,LPTT} and prove that for any $0<p<1,$ there exists a martingale $f\in H_{p}$ such that 
$
\sup_{n\in \mathbb{N}} \left\Vert L_{2^n}f\right\Vert_{weak-L_p}=\infty.
$
Moreover, in our Theorem 1 we generalize the result of Goginava \cite{gog2} and prove that for any $f\in L_1(G)$ and for any Lebesgue point $x$,
\begin{equation*}
L_{2^n}f(x)\to f(x), \ \ \ \text{as} \ \ \  n\to\infty.
\end{equation*}

The main results in this paper are presented and proved in Section 4. Section 3 is used to present some auxiliary lemmas, where, in particular, Lemma 2 is new and of independent interest. In order not to disturb our discussions later on some definitions and notations are given in Section 4. Finally, Section 5 is reserved for some open questions we hope can be a source of inspiration for further research in this interesting area.

\section{Definitions and Notations}

\bigskip Let $\mathbb{N}_{+}$ denote the set of the positive integers, $%
\mathbb{N}:=\mathbb{N}_{+}\cup \{0\}.$ Denote by $Z_{2}$ the discrete cyclic
group of order 2, that is $Z_{2}:=\{0,1\},$ where the group operation is the
modulo 2 addition and every subset is open. The Haar measure on $Z_2$ is
given so that the measure of a singleton is 1/2.

Define the group $G$ as the complete direct product of the group $Z_{2},$
with the product of the discrete topologies of $Z_{2}$`s. The elements of $G$
are represented by sequences $x:=(x_{0},x_{1},...,x_{j},...),$ where $%
x_{k}=0\vee 1.$

It is easy to give a base for the neighborhood of $x\in G$ namely:
\begin{equation*}
I_{0}\left( x\right) :=G,\text{ \ }I_{n}(x):=\{y\in
G:y_{0}=x_{0},...,y_{n-1}=x_{n-1}\}\text{ }(n\in \mathbb{N}).
\end{equation*}

Denote $I_{n}:=I_{n}\left( 0\right) ,$ $\overline{I_{n}}:=G$ $\backslash $ $%
I_{n}$ and $e_{n}:=\left( 0,...,0,x_{n}=1,0,...\right) \in G,$ for $n\in 
\mathbb{N}$. It is easy to show that
$
\overline{I_{M}}=\overset{M-1}{\underset{s=0}{\bigcup }}I_{s}\backslash
I_{s+1}.  $

If $n\in \mathbb{N},$ then every $n$ can be uniquely expressed as $%
n=\sum_{k=0}^{\infty }n_{j}2^{j},$ where $n_{j}\in Z_{2}$ $~(j\in \mathbb{N}%
) $ and only a finite numbers of $n_{j}$ differ from zero. Let $%
\left\vert n\right\vert :=\max \{k\in \mathbb{N}:\ n_{k}\neq 0\}.$

The norms (or quasi-norms) of the spaces $L_{p}(G)$ and $weak-L_{p}\left(
G\right) ,$ $\left( 0<p<\infty \right) $ are, respectively, defined by 
\begin{equation*}
\left\Vert f\right\Vert _{p}^{p}:=\int_{G}\left\vert f\right\vert ^{p}d\mu
,\ \ \ \ \left\Vert f\right\Vert_{weak-L_{p}}^{p}:=\sup_{\lambda
>0}\lambda ^{p}\mu \left( f>\lambda \right) .
\end{equation*}

The $k$-th Rademacher function is defined by%
\begin{equation*}
r_{k}\left( x\right) :=\left( -1\right) ^{x_{k}}\text{\qquad }\left( \text{ }%
x\in G,\text{ }k\in \mathbb{N}\right) .
\end{equation*}

Now, define the Walsh system $w:=(w_{n}:n\in \mathbb{N})$ on $G$ as: 
\begin{equation*}
w_{n}(x):=\overset{\infty }{\underset{k=0}{\Pi }}r_{k}^{n_{k}}\left(
x\right) =r_{\left\vert n\right\vert }\left( x\right) \left( -1\right) ^{%
\underset{k=0}{\overset{\left\vert n\right\vert -1}{\sum }}n_{k}x_{k}}\text{%
\qquad }\left( n\in \mathbb{N}\right) .
\end{equation*}

It is well-known that  (see
e.g. \cite{sws})
\begin{eqnarray}\label{vilenkin}
w_n\left( x+y\right) &=&w_n\left( x\right)w_n\left( y\right).
\end{eqnarray}

The Walsh system is orthonormal and complete in $L_{2}\left( G\right) $ (see
e.g. \cite{sws}).

If $f\in L_{1}\left( G\right) $ we can establish  Fourier coefficients,
 partial sums of the Fourier series,  Dirichlet kernels with respect
to the Walsh system in the usual manner:
\begin{equation*}
\widehat{f}\left( k\right) :=\int_{G}fw_{k}d\mu \,\,\,\,\left( k\in \mathbb{N
}\right) ,
\end{equation*}%
\begin{equation*}
S_{n}f:=\sum_{k=0}^{n-1}\widehat{f}\left( k\right) w_{k},\ \
D_{n}:=\sum_{k=0}^{n-1}w_{k\text{ }}\,\,\,\left( n\in \mathbb{N}_{+}\right) .
\end{equation*}

Recall that (for details see e.g. \cite{AVD})
\begin{equation}
D_{2^{n}}\left( x\right) =\left\{ 
\begin{array}{ll}
2^{n}, & \,\text{if\thinspace \thinspace \thinspace }x\in I_{n} \\ 
0, & \text{if}\,\,x\notin I_{n}%
\end{array}%
\right.  \label{1dn}
\end{equation}%
and
\begin{equation}
D_{n}=w_{n}\overset{\infty }{\underset{k=0}{\sum }}n_{k}r_{k}D_{2^{k}}=w_{n}%
\overset{\infty }{\underset{k=0}{\sum }}n_{k}\left(
D_{2^{k+1}}-D_{2^{k}}\right) ,\text{ for \ }n=\overset{\infty }{\underset{i=0%
}{\sum }}n_{i}2^{i}.  \label{2dn}
\end{equation}

Let $\left\{ q_{k}:k\geq 0\right\} $ be a sequence of nonnegative numbers.
The Nörlund means for the Fourier series of $f$ are defined by 
\begin{equation*}
\frac{1}{{l_{n}}}\sum_{k=0}^{n}q_{n-k}S_{k}f.
\end{equation*}

In the special case when $\{q_{k}=1:k\in \mathbb{N}\},$ we get the Fej\'er means
\begin{equation*}
\sigma _{n}f:=\frac{1}{n}\sum_{k=1}^{n}S_{k}f\,.
\end{equation*}%
$\,$

If $q_{k}={1}/{(k+1)}$, then we get the Nörlund logarithmic means: 
\begin{equation}
L_{n}f:=\frac{1}{l_{n}}\sum_{k=0}^{n-1}\frac{S_{k}f}{n-k}, \ \ \ \ l_{n}:=\sum_{k=1}^{n}%
\frac{1}{k}.  \label{L_n}
\end{equation}

The $n$-th Riesz logarithmic mean of the Fourier
series of the integrable function $f$ is defined by 
\begin{equation*}
R_{n}f:=\frac{1}{l_{n}}\sum_{k=1}^{n}\frac{S_{k}f}{k}, \ \ l_{n}:=\sum_{k=1}^{n}%
\frac{1}{k},
\end{equation*}%
We note that it is an  inverse of the Nörlund logarithmic means.

The convolution of two functions
$f,g\in L_{1}(G)$ is defined by
\begin{equation*}
\left( f\ast g\right) \left( x\right) :=\int_{G}f\left( x+t\right)
g\left( t\right) d\mu (t)\text{ \ \ }\left( x\in G\right).
\end{equation*}

It is well-known that if $f\in L_{p}\left( G\right) ,$ $g\in
L_{1}\left(  G\right) $ and $1\leq p<\infty .$ Then $f\ast g\in
L_{p}\left(  G\right) $ and the corresponding inequality holds:
\begin{equation} \label{covstrongaaaj}
\left\Vert f\ast g\right\Vert _{p}\leq \left\Vert f\right\Vert_{p}\left\Vert g\right\Vert _{1}.
\end{equation}

The representations
\begin{equation*}
L_nf\left(x\right)=\underset{G}{\int}f\left(t\right)P_n\left(x+t\right) d\mu\left(t\right)
\ \ \text{and} \ \
R_nf\left(x\right)=\underset{G}{\int}f\left(t\right)Y_n\left(x+t\right) d\mu\left(t\right)
\end{equation*}
for $n\in \mathbb{N}$ play a central role in the sequel, where 
\begin{equation*}
P_n:=\frac{1}{Q_n}\overset{n}{\underset{k=1}{\sum }}q_{n-k}D_k
\ \ \text{and} \ \
Y_n:=\frac{1}{Q_n}\overset{n}{\underset{k=1}{\sum }}q_{k}D_k
\end{equation*}
are called the kernels of the N\"orlund  logaritmic and the Reisz means, respectively.

It is well-known that (see e.g. Goginava \cite{gog2} and Tephnadze \cite{tepthesis}):
\begin{eqnarray} \label{1.71alphaT2j} P_{2^n}(x)=D_{2^n}(x)-\psi_{2^n-1}(x){Y}_{2^n}(x).
\end{eqnarray}
Moreover, for all $n\in \mathbb{N},$
\begin{eqnarray} \label{1.71alphaT2jj} 
\Vert P_{2^n}\Vert_1<c<\infty \ \ \ \text{and} \ \ \
\Vert Y_{n}\Vert_1<c<\infty.
\end{eqnarray}

In the case $f\in L_{1}(G)$ the maximal functions are given by  
\begin{equation*}
M(f)(x)=\sup_{n\in\mathbb{N}}\frac{1}{\left\vert I_{n}\left( x\right)
	\right\vert }\left\vert \int_{I_{n}\left( x\right) }f\left( u\right)d \mu
\left( u\right) \right\vert =\sup_{n\in\mathbb{N}} 2^n \left\vert \int_{I_{n}\left( x\right) }f\left( u\right) d\mu
\left( u\right) \right\vert .
\end{equation*}

It is well-known (for details see e.g. \cite{sws} and \cite{Tor1}) that if $f\in L_1(G),$ then the inequality
\begin{equation*}
\Vert M(f)\Vert_{\text{weak}-L_1}\leq \Vert f\Vert_{1}.
\end{equation*}
holds. 
According to a density argument of Calderon-Zygmund (see \cite{Tor1}) we obtain that if  $f\in L_{1}\left( G\right),$ then 
\begin{equation*}\label{a.e.maxfunct}
2^n \left\vert \int_{I_{n}\left( x\right) }f\left( u\right) d\mu
\left( u\right) \right\vert \to 0, \ \ \text{as} \ \ n\to \infty.
\end{equation*}

A point $x$ on the Walsh group is called a Lebesgue point of   $f\in L_{1}\left( G\right),$ if

\[\lim_{n\rightarrow \infty }2^n\int_{I_{n}(x)}f\left(t\right) d\mu(t)=f\left( x\right) \ \ \ \ \ a.e.\ \ \ x\in G.\]

According to \eqref{a.e.maxfunct} we find that if  $f\in L_{1}\left( G\right),$ then a.e point is a Lebesgue point.

The $\sigma $-algebra generated by the intervals $\left\{ I_{n}\left(
x\right) :x\in G\right\} $ is denoted by $\digamma _{n}\left( n\in 
\mathbb{N}\right) .$ 
Let $f:=\left( f^{\left( n\right) },n\in \mathbb{N}%
\right) $ be a martingale with respect to $\digamma _{n}\left( n\in \mathbb{N%
}\right) $ (for details see e.g. \cite{We1}).

The maximal function of a martingale $f$ \ is defined by 
\begin{equation*}
f^{\ast }:=\sup_{n\in \mathbb{N}}\left\vert f^{(n)}\right\vert .
\end{equation*}

For $0<p<\infty $ the Hardy martingale spaces $H_{p}\left( G\right) $
consist of all martingales for which
\begin{equation*}
\left\| f\right\| _{H_{p}}:=\left\| f^{*}\right\|_{p}<\infty .
\end{equation*}

If $f\in L_{1}\left( G\right),$ then it is easy to show that the sequence $F=
\left( S_{2^{n}}f :n\in \mathbb{N}\right) $ is a martingale and $F^*=M(f).$

If $f=\left( f^{\left( n\right) },\text{ }n\in \mathbb{N}\right) $ is a
martingale, then the Walsh-Fourier coefficients must be defined in a slightly
different manner:
\begin{equation*}
\widehat{f}\left( i\right) :=\lim_{k\rightarrow \infty }\int_{G}f^{\left(
k\right) }\left( x\right)w_{i}\left( x\right) d\mu \left( x\right) .
\end{equation*}

The Walsh-Fourier coefficients of $f\in L_{1}\left( G\right) $ are the same
as those of the martingale $\left( S_{2^n}f :n\in \mathbb{N}%
\right) $ obtained from $f$.

A bounded measurable function $a$ is a $p$-atom if there exists an interval $I$
such that \qquad 
\begin{equation*}
\int_{I}ad\mu =0,\ \left\Vert a\right\Vert _{\infty }\leq \mu \left(
I\right) ^{-1/p} ,\text{ \ supp}\left( a\right) \subset I.
\end{equation*}

\section{Auxiliary Results}

The Hardy martingale space $H_{p}\left( G\right) $ has an atomic
characterization (see Weisz \cite{We1}, \cite{We3}):

\begin{lemma} \label{Weisz}
A martingale $f=\left( f^{\left( n\right) },\ n\in \mathbb{N}
\right) $ is in $H_{p}\left( 0<p\leq 1\right) $ if and only if there exist a
sequence $\left( a_{k},k\in \mathbb{N}\right) $ of p-atoms and a sequence $%
\left( \mu _{k},k\in\mathbb{N}\right) $ of real numbers such that for
every $n\in \mathbb{N}:$

\begin{equation}\label{1}
\underset{k=0}{\overset{\infty }{\sum }}\mu _{k}S_{2^{n}}a_{k}=f^{\left(
n\right) },   \ \ \ \ \text{where} \ \ \ \ 
\sum_{k=0}^{\infty }\left| \mu _{k}\right| ^{p}<\infty . 
\end{equation}

Moreover,
\begin{equation*}
\left\| f\right\| _{H_{p}}\backsim \inf \left( \sum_{k=0}^{\infty }\left|
\mu _{k}\right| ^{p}\right) ^{1/p},
\end{equation*}
where the infimum is taken over all decompositions of $f$ of the form (\ref{1}).
\end{lemma}

We also state and prove the following new lemma of independent interest:
\begin{lemma}
Let $n\in \mathbb{N}$  and $x\in I_{2}(e_{0}+e_{1})\in I_{0}\backslash I_{1}.$ Then
\begin{eqnarray*}
\left\vert\sum_{j=2^{2\alpha_k}}^{2^{2\alpha _k+1}}\frac{D_j}{2^{2\alpha _k+1}-j}
\right\vert 
&=&\left\vert\sum_{j=2^{2\alpha _{k}-1}+1}^{2^{2\alpha _{k}}-1}\frac{
	w_{2j+1}}{2^{2\alpha _{k}+1}-2j-1}\right\vert  \\
&=&\left\vert \sum_{j=2^{2\alpha _{k}-1}+1}^{2^{2\alpha _{k}}-1}\frac{w_{2j}}{2^{2\alpha _{k}+1}-2j-1}\right\vert  \geq \frac{1}{3}.
\end{eqnarray*}
\end{lemma}
\begin{proof}
Let $x\in I_{2}(e_{0}+e_{1})\in I_{0}\backslash I_{1}.$ According to (\ref%
{1dn}) and (\ref{2dn}) we can conclude that
\begin{equation*}
D_{j}\left( x\right) =\left\{ 
\begin{array}{ll}
w_{j}, & \,\text{if\thinspace \thinspace \thinspace }j\ \ \text{is odd
	number,} \\ 
0, & \text{if}\,\,j\ \ \text{is even number,}%
\end{array}%
\right.
\end{equation*}
and 
\begin{eqnarray*}
&&\sum_{j=2^{2\alpha_k}}^{2^{2\alpha _k+1}-1}\frac{D_j}{2^{2\alpha _k+1}-j}\\
	&=&\sum_{j=2^{2\alpha _{k}-1}}^{2^{2\alpha _{k}}-1}\frac{%
		w_{2j+1}}{2^{2\alpha _{k}+1}-2j-1}=w_1\sum_{j=2^{2\alpha _{k}-1}}^{2^{2\alpha _{k}}-1}\frac{w_{2j}}{%
		2^{2\alpha _{k}+1}-2j-1}.
\end{eqnarray*}

Since
\begin{eqnarray*}
	&&\sum_{j=2^{2\alpha _{k}-2}+1}^{2^{2\alpha _{k}-1}-1}\left\vert \frac{1}{%
		2^{2\alpha _{k}+1}-4j+3}-\frac{1}{2^{2\alpha _{k}+1}-4j+1}\right\vert \\
	&=&\sum_{j=2^{2\alpha _{k}-2}+1}^{2^{2\alpha _{k}-1}-1}\frac{2}{(2^{2\alpha
			_{k}+1}-4j+3)(2^{2\alpha _{k}+1}-4j+1)} \\
	&\leq &\sum_{j=2^{2\alpha _{k}-2}+1}^{2^{2\alpha _{k}-1}-1}\frac{2}{%
		(2^{2\alpha _{k}+1}-4j)(2^{2\alpha _{k}+1}-4j)} \\
	&\leq &\frac{1}{8}\sum_{j=2^{2\alpha _{k}-2}+1}^{2^{2\alpha _{k}-1}-1}\frac{1%
	}{(2^{2\alpha _{k}-1}-j)(2^{2\alpha _{k}-1}-j)} \\
&\leq &\frac{1}{8}\sum_{k=1}^{\infty }\frac{1}{k^2}\leq \frac{1}{8}+ \frac{1}{8}\sum_{k=2}^{\infty }\frac{1}{k^2} \\
&\leq&\frac{1}{8}+\frac{1}{8} \sum_{k=2}^{\infty}\left(\frac{1}{k-1}-\frac{1}{k}\right) \leq \frac{1}{8}+\frac{1}{8}=\frac{1}{4},
\end{eqnarray*}
according to 
$$
w_{4k+2}=w_{2}w_{4k}=-w_{4k}, \ \ \text{for} \ \ x\in I_{2}(e_{0}+e_{1}),
$$
we can conclude that
\begin{eqnarray*}
&&\left\vert w_{2^{2\alpha_{k}+1}-2}+\frac{w_{2^{2\alpha_{k}+1}-4}}{3}+\sum_{j=2^{2\alpha _{k}-1}+1}^{2^{2\alpha _{k}}-1}\frac{w_{2j}}{2^{2\alpha _{k}+1}-2j-1}\right\vert   \\
&=& \left\vert \frac{w_{2^{2\alpha_{k}+1}-4}}{3}-
w_{2^{2\alpha _{k}+1}-4}+\sum_{j=2^{2\alpha_{k}-2}+1}^{2^{2\alpha _{k}-1}}\left( \frac{w_{4j-4}}{2^{2\alpha _{k}+1}-4j+3}+\frac{w_{4j-2}}{2^{2\alpha _{k}+1}-4j+1}\right) \right\vert  \\
&=&\left\vert \frac{2w_{2^{2\alpha _{k}+1}-4}}{3} -\sum_{j=2^{2\alpha_{k}-2}+1}^{2^{2\alpha _{k}-1}}\left( \frac{w_{4j-4}}{2^{2\alpha _{k}+1}-4j+3}-\frac{w_{4j-4}}{2^{2\alpha _{k}+1}-4j+1}\right) \right\vert  \\
&\geq & \frac{2}{3}-\sum_{j=2^{2\alpha _{k}-2}+1}^{2^{2\alpha
			_{k}-1}}\left\vert \frac{1}{2^{2\alpha _{k}+1}-4j+3}-\frac{1}{2^{2\alpha
			_{k}+1}-4j+1}\right\vert  \\
	&\geq & \frac{2}{3}-\frac{1}{4}=\frac{5}{12}\geq \frac{1}{3}.
\end{eqnarray*}

The proof is complete.
\end{proof}

\section{Main results}

Our first main result reads:
\begin{theorem}\label{Corollaryconv4} Let $p\geq 1$ and $f\in L_p(G)$. Then
\begin{eqnarray}\label{normconverg}
\Vert L_{2^n} f-f\Vert_p \to 0 \ \ \text{as}\ \ n\to \infty.
\end{eqnarray}
Moreover,
	\begin{equation}\label{jojo}
	\underset{n\rightarrow \infty }{\lim }L_{2^n}f(x)=f(x),
	\end{equation}
	for all Lebesgue points of $f$.
\end{theorem}

\begin{proof}\label{lemma0nnT121}Let $n\in \mathbb{N}$.  By combining \eqref{covstrongaaaj}  and \eqref{1.71alphaT2jj} we immediately get
	
$$\Vert L_{2^n} f\Vert_p \leq c_p\Vert f\Vert_p \ \ \text{for all}\ \ n\in \mathbb{N},$$
which immediately implies \eqref{normconverg}.

To prove a.e convergence we use identity \eqref{1.71alphaT2j} to obtain that
\begin{eqnarray}\label{jojo1}
L_{2^n}f\left(x\right)&=&\underset{G}{\int}f\left(t\right)P_{2^n}\left(x+t\right) d\mu\left(t\right)\\ \notag
&=&\underset{G}{\int}f\left(t\right)D_{2^n}\left(x+t\right)d\mu\left(t\right)\\ \notag
&-&\underset{G}{\int}f\left(t\right)w_{2^n-1}(x+t){Y}_{2^n}(x+t)\\	\notag
&:=&I-II. 
\end{eqnarray}

By applying  \eqref{s2nae} we can conclude that

\begin{eqnarray} \label{jojo2}
I=S_{2^n}f(x)\to f(x)
\end{eqnarray}
for all Lebesgue points of $f\in L_p(G)$. Moreover, by using \eqref{vilenkin} we find that

\begin{eqnarray*}
II=\psi_{2^n-1}(x)\underset{G}{\int}f\left(t\right){Y}_{2^n}(x+t){\psi}_{2^n-1}(t)d(t).
\end{eqnarray*}

In view of \eqref{1.71alphaT2jj} we see that  

$$f\left(t\right){Y}_{2^n}(x+t)\in L_p \ \ \text{ where} \ \  p\geq  1 \ \ \text{for any } \ \ x\in G, $$
and also note that $II$ describes the Fourier coefficients of an integrable function. Hence, according to the Riemann-Lebesgue Lemma it vanishes as $n\to \infty,$ i.e.
\begin{eqnarray}\label{jojo3}
II\to 0 \ \ \text{for any } \ \ x\in G, \ \  n\to \infty.
\end{eqnarray}
The proof of \eqref{jojo} follows by just combining \eqref{jojo1}-\eqref{jojo3}.

The proof is complete.

\end{proof}

Our next main result is the following answer of question 1:
\begin{theorem}
Let $0<p<1.$ Then there exists a martingale $f\in H_{p}$ such that 
\begin{equation*}
\sup_{n\in \mathbb{N}} \left\Vert L_{2^n}f\right\Vert_{weak-L_p}=\infty.
\end{equation*}
\end{theorem}

\begin{proof}
Let $\left\{ \alpha _{k}:k\in \mathbb{N}\right\} $ be an increasing sequence of the
positive integers such that

\qquad 
\begin{equation}
\sum_{k=0}^{\infty }\alpha _{k}^{-p/2}<\infty ,  \label{3}
\end{equation}

\begin{equation}
\sum_{\eta =0}^{k-1}\frac{\left( 2^{2\alpha _{\eta }}\right) ^{1/p}}{\sqrt{
\alpha _{\eta }}}<\frac{\left( 2^{2\alpha _{k}}\right) ^{1/p}}{\sqrt{\alpha
_{k}}},  \label{4}
\end{equation}

\begin{equation}
\frac{\left( 2^{2\alpha _{k-1}}\right) ^{1/p}}{\sqrt{\alpha _{k-1}}}<\frac{%
2^{2\alpha _{k}-8}}{\alpha _{k}^{1/2}l_{2^{2\alpha _{k}+1}}}.  \label{5}
\end{equation}

We note that such an increasing sequence $\left\{ \alpha _{k}:k\in \mathbb{N}%
\right\} $ which satisfies conditions (\ref{3})-(\ref{5}) can obviously be constructed.

Let 
\begin{equation*}
f^{\left( n\right) }\left( x\right) :=\sum_{\left\{ k;\text{ }2\alpha
_{k}<n\right\} }\lambda _{k}a_{k},
\end{equation*}
where 
\begin{equation*}
\lambda _{k}=\frac{1}{\sqrt{\alpha _{k}}}
\end{equation*}
and
\begin{equation*}
a_{k}={2^{2\alpha _{k}(1/p-1)}}\left( D_{2^{2\alpha _{k}+1}}-D_{2^{2\alpha
_{k}}}\right) .
\end{equation*}

From (\ref{3}) and Lemma \ref{Weisz} we can conclude that $f=\left( f^{\left(
n\right) },n\in \mathbb{N}\right) \in H_{p}.$

It is easy to show that
\begin{equation}
\widehat{f}(j)=\left\{ 
\begin{array}{l}
\frac{2^{2\alpha _{k}(1/p-1)}}{\sqrt{\alpha _{k}}},\,\,\text{ if \thinspace
\thinspace }j\in \left\{ 2^{2\alpha _{k}},...,\text{ ~}2^{2\alpha
_{k}+1}-1\right\} ,\text{ }k\in \mathbb{N}, \\ 
0,\text{ \thinspace \thinspace \thinspace if \thinspace \thinspace
\thinspace }j\notin \bigcup\limits_{k=1}^{\infty }\left\{ 2^{2\alpha
_{k}},...,\text{ ~}2^{2\alpha _{k}+1}-1\right\}\text{ .}
\end{array}%
\right.  \label{8}
\end{equation}

Moreover,
\begin{eqnarray}
&&L_{2^{2\alpha _{k}+1}}f  \label{10a} \\
&=&\frac{1}{l_{2^{2\alpha _{k}+1}}}\sum_{j=1}^{2^{2\alpha _{k}}-1}\frac{%
S_{j}f}{2^{2\alpha _{k}+1}-j}+\frac{1}{l_{2^{2\alpha _{k}+1}}}%
\sum_{j=2^{2\alpha _{k}}}^{2^{2\alpha _{k}+1}-1}\frac{S_{j}f}{{2^{2\alpha
_{k}+1}-j}}  \notag \\
&:=&I+II.  \notag
\end{eqnarray}

Let $j<2^{2\alpha _{k}}.$ Then from (\ref{4}), (\ref{5}) and (\ref{8}) we
have
\begin{eqnarray*}
&&\left\vert S_{j}f\left( x\right) \right\vert \\
&\leq &\sum_{\eta =0}^{k-1}\sum_{v=2^{2\alpha _{\eta }}}^{2^{2\alpha _{\eta
}+1}-1}\left\vert \widehat{f}(v)\right\vert \leq \sum_{\eta
=0}^{k-1}\sum_{v=2^{2\alpha _{\eta }}}^{2^{2\alpha _{\eta }+1}-1}\frac{%
2^{2\alpha _{\eta }(1/p-1)}}{\sqrt{\alpha _{\eta }}} \\
&\leq &\sum_{\eta =0}^{k-1}\frac{2^{2\alpha _{\eta }/p}}{\sqrt{\alpha _{\eta
}}}\leq \frac{2^{2\alpha _{k-1}/p+1}}{\sqrt{\alpha _{k-1}}}<\frac{2^{2\alpha
_{k}-4}}{\alpha _{k}^{1/2}l_{2^{2\alpha _{k}+1}}}.
\end{eqnarray*}

Consequently,
\begin{eqnarray}
\left\vert I\right\vert &\leq &\frac{1}{l_{2^{2\alpha _{k}+1}}}\underset{j=1}%
{\overset{2^{2\alpha _{k}}-1}{\sum }}\frac{\left\vert S_{j}f\left( x\right)
\right\vert }{2^{2\alpha _{k}+1}-j}  \label{11a} \\
&\leq &\frac{1}{l_{2^{2\alpha _{k}+1}}}\frac{2^{2\alpha _{k-1}/p}}{\sqrt{%
\alpha _{k-1}}}\sum_{j=1}^{M_{2\alpha _{k}+1}-1}\frac{1}{j}\leq\frac{2^{2\alpha _{k-1}/p}}{\sqrt{\alpha _{k-1}}}.  \notag
\end{eqnarray}

Let $2^{2\alpha _{k}}\leq j\leq 2^{2\alpha_{k}+1}-1.$ Then we have the
following equality

\begin{eqnarray*}
S_{j}f &=&\sum_{\eta =0}^{k-1}\sum_{v=2^{2\alpha _{\eta }}}^{2^{2\alpha
_{\eta }+1}-1}\widehat{f}(v)w_{v}+\sum_{v=2^{2\alpha _{k}}}^{j-1}\widehat{f}%
(v)w_{v} \\
&=&\sum_{\eta =0}^{k-1}\frac{2^{{2\alpha _{\eta }}\left( 1/p-1\right) }}{%
\sqrt{\alpha _{\eta }}}\left( D_{2^{2\alpha _{\eta }+1}}-D_{2^{2\alpha
_{\eta }}}\right) \\
&&+\frac{2^{{2\alpha _{k}}\left( 1/p-1\right) }}{\sqrt{\alpha _{k}}}\left(
D_{j}-D_{2^{{2\alpha _{k}}}}\right) .
\end{eqnarray*}

This gives that
\begin{equation}
II=\frac{1}{l_{2^{2\alpha _{k}+1}}}\underset{j=2^{2\alpha _{k}}}{\overset{%
2^{2\alpha _{k}+1}}{\sum }}\ \frac{1}{2^{2\alpha _{k}+1}-j}\left( \sum_{\eta
=0}^{k-1}\frac{2^{2\alpha _{\eta }\left( 1/p-1\right) }}{\sqrt{\alpha _{\eta
}}}\left( D_{2^{2\alpha _{\eta }+1}}-D_{2^{_{2\alpha _{\eta }}}}\right)
\right)  \label{12a}
\end{equation}
\begin{eqnarray*}
&+&\frac{1}{l_{2^{2\alpha _{k}+1}}}\frac{2^{2\alpha _{k}\left( 1/p-1\right) }}{%
\sqrt{\alpha_k}}\sum_{j=2^{2\alpha_k}}^{2^{2\alpha _{k}+1}-1}\frac{
\left(D_j-D_{2^{2\alpha_k}}\right)}{2^{2\alpha _{k}+1}-j}
:=II_{1}+II_{2}.
\end{eqnarray*}

Let $x\in I_{2}(e_0+e_1)\in I_0\backslash I_1.$ Since $\alpha _{0}\geq 1$ we obtain that $2\alpha_k\geq 2,$ for all $k\in \mathbb{N}$ and if we apply (\ref{1dn}) we get
that $D_{2^{2\alpha_k}}=0,$ 
\begin{equation}
II_{1}=0  \label{13a}
\end{equation}
and 
\begin{eqnarray*}
&&II_{2}=\frac{1}{l_{2^{2\alpha _{k}+1}}}\frac{2^{2\alpha _{k}(1/p-1)}}{%
\sqrt{\alpha _{k}}}\sum_{j=2^{2\alpha _{k}-1}+1}^{2^{2\alpha _{k}}-1}\frac{%
w_{2j+1}}{2^{2\alpha _{k}+1}-2j-1} \\
&=&\frac{1}{l_{2^{2\alpha _{k}+1}}}\frac{2^{2\alpha _{k}(1/p-1)}w_{1}}{\sqrt{%
\alpha _{k}}}\sum_{j=2^{2\alpha _{k}-1}+1}^{2^{2\alpha _{k}}-1}\frac{w_{2j}}{%
2^{2\alpha _{k}+1}-2j-1}.
\end{eqnarray*}

By now using Lemma 2 we can conclude that

\begin{equation} \label{14a}
\left\vert II_{2}\right\vert \geq \frac{1}{3}\frac{1}{l_{2^{2\alpha _{k}+1}}}\frac{%
2^{2\alpha _{k}(1/p-1)}}{\sqrt{\alpha _{k}}}\geq \frac{1}{l_{2^{2\alpha _{k}+1}}%
}\frac{2^{2\alpha _{k}(1/p-1)-1}}{\sqrt{\alpha _{k}}}.
\end{equation}%

By combining (\ref{5}), (\ref{10a})-(\ref{14a}) for $x\in I_{2}(e_{0}+e_{1})$
and $0<p<1$ we have that

\begin{eqnarray*}
\left\vert L_{2^{2\alpha_k+1}}f\left(x\right)\right\vert &\geq& II_{2}-II_{1}-I \\
&\geq &\frac{1}{l_{2^{2\alpha _{k}+1}}}\frac{2^{2\alpha _{k}(1/p-1)-2}}
{\sqrt{\alpha _{k}}}-\frac{1}{l_{2^{2\alpha _{k}+1}}}\frac{2^{2\alpha
_{k}(1/p-1)-3}}{\sqrt{\alpha _{k}}} \\
&\geq &\frac{1}{l_{2^{2\alpha _k+1}}}\frac{2^{2\alpha_k(1/p-1)-3}}
{\sqrt{\alpha_k}}\geq \frac{2^{2\alpha _{k}(1/p-1)-3}}{\left(\ln 2^{2\alpha_k+1}+1\right)\sqrt{\alpha_k}} \\
&\geq& \frac{2^{2\alpha _{k}(1/p-1)-3}}{( 4\alpha _{k}+1)\sqrt{\alpha _{k}}}\\
&\geq& \frac{2^{2\alpha _{k}(1/p-1)-6}}{\alpha^{3/2}_k}.
\end{eqnarray*}

Hence, we can conclude that
\begin{eqnarray*}
&&\left\Vert L_{q_{\alpha _{k}}^{s}}f\right\Vert _{weak-L_{p}}  \\
&\geq &\frac{2^{2\alpha _{k}(1/p-1)-6}}{\alpha^{3/2}_k}\mu \left\{ x\in G:\left\vert L_{2^{2\alpha
_{k}+1}}f\right\vert \geq \frac{2^{2\alpha _{k}(1/p-1)-6}}{\alpha^{3/2}_k}\right\} ^{1/p} \\
&\geq &\frac{2^{2\alpha _{k}(1/p-1)-6}}{\alpha^{3/2}_k}\mu \left\{ x\in I_{2}(e_{0}+e_{1}):\left\vert
L_{2^{2\alpha _{k}+1}}f\right\vert \geq \frac{2^{2\alpha _{k}(1/p-1)-6}}{\alpha^{3/2}_k}\right\} ^{1/p} \\
&\geq &\frac{2^{2\alpha _{k}(1/p-1)-6}}{\alpha^{3/2}_k}(\mu \left( x\in I_{2}(e_{0}+e_{1})\right) )^{1/p} \\
&>&\frac{c2^{2\alpha _{k}(1/p-1)}}{\alpha _{k}^{3/2}}\rightarrow \infty ,%
\text{ \ \ as \ }k\rightarrow \infty .
\end{eqnarray*}

The proof is complete.
\end{proof}

\section{Open questions}

It is known (for details see e.g. the books \ \cite{sws}
and \cite{We1}) that the subsequence $\{S_{2^{n}}\}$ of the partial
sums is bounded from the martingale Hardy space $H_{p}$ to the Lebesgue
space $L_{p},$ for all $p>0.$ On the other hand, (see Tephnadze \cite{tep7}) there
exists a martingale $f\in H_{p}$ $\left( 0<p<1\right) ,$ such that

\begin{equation*}
\underset{n\in \mathbb{N}}{\sup }\left\Vert S_{2^{n}+1}f\right\Vert
_{weak-L_{p}}=\infty .
\end{equation*}%
However, Simon \cite{si11} proved that for all $f\in H_{p},$ there exists
an absolute constant $c_{p},$ depending only on $p,$ such that 

\begin{equation*}
\overset{\infty }{\underset{k=1}{\sum }}\frac{\left\Vert S_{k}f\right\Vert
	_{p}^{p}}{k^{2-p}}\leq c_{p}\left\Vert f\right\Vert _{H_{p}}^{p},\text{ \ \
	\ }\left( 0<p<1\right).
\end{equation*}%
In \cite{TT1} it was proved that for all $f\in H_{p},$ there exists
an absolute constant $c_{p},$ depending only on $p,$ such that 

\begin{equation*}
\overset{\infty }{\underset{k=1}{\sum }}\frac{\left\Vert L_{k}f\right\Vert
	_{p}^{p}}{k^{2-p}}\leq c_{p}\left\Vert f\right\Vert _{H_{p}}^{p},\text{ \ \
	\ }\left( 0<p<1\right).
\end{equation*}%

\textbf{Open Problem 1:} a) Let $f\in H_{p},$  where $0<p<1.$ Does there exist
an absolute constant $c_{p},$ depending only on $p,$ such that the following inequality holds:

\begin{equation*}
\overset{\infty }{\underset{k=1}{\sum }}\frac{\log^p k\left\Vert L_{k}f\right\Vert
	_{p}^{p}}{k^{2-p}}\leq c_{p}\left\Vert f\right\Vert _{H_{p}}^{p},\text{ \ \
	\ }\left( 0<p<1\right)?
\end{equation*}

b) For $0<p<1/2$ and any non-decreasing function $\Phi :\mathbb{N}\rightarrow \lbrack 1,\infty )$ satisfying the conditions 

$$\underset{n\rightarrow \infty }{\lim }\Phi \left( n\right) =+\infty,$$
is it possible to find a martingale $f\in H_p\left(G_m\right) $ such that

\begin{equation*}
\sum_{n=1}^{\infty}\frac{\log^p n \left\Vert L_nf\right\Vert_p^p\Phi\left(n\right)}{n^{2-p}} =\infty?
\end{equation*}

\textbf{Open Problem 2:} a) Let $f\in H_{p}$ where $0<p\leq 1$ and

\begin{equation*}
\omega _{H_p}\left( \frac{1}{M_{n}},f\right) =o\left( \frac{\log n}{2^{n(1/p-1)}\log^{2[p]}n}\right), \ \ \ \text{ as } \ \ \ n\rightarrow \infty.
\end{equation*}

Does the following convergence result hold:
\begin{equation*}
\left\Vert L_{k}f-f\right\Vert _{H_{p}}\rightarrow 0,\,\,\,\text{as \ }%
\,\,\,k\rightarrow \infty?
\end{equation*}

b) Let $0<p\leq 1.$ Does there exist a martingale $f\in H_{p}$, for which

\begin{equation*}
\omega _{H_{p}}\left( \frac{1}{2^{n}},f\right) =O\left( \frac{\log n}{2^{n(1/p-1)}\log^{2[p]}n}\right) ,\text{ \ as \  }n\rightarrow \infty
\end{equation*}%
and

\begin{equation*}
\left\Vert L_{k}f-f\right\Vert _{weak-L_{p}}\nrightarrow 0,\,\,\,\text{as\thinspace \thinspace \thinspace }k\rightarrow \infty ?
\end{equation*}

Reisz logarithmic means have better approximation properties then the Fej\'er means and if it is converging in some sense then the Fej\'er means converge in the same sense. Moreover, it has similar boundedness properties as the Fej\'er means  when we consider $(H_p,L_p)$ and $(H_p,weak-L_p)$ type inequalities for the  maximal operators of Riesz logarithmic means for $0<p\leq 1.$  In particular, it was proved in \cite{tep20} that the maximal operator of the Riesz logarithmic means of Vilenkin-Fourier series is bounded from the Hardy space $H_{1/2}(G_m)$ to the space $weak-L_{1/2}(G_m).$ It follows that it is bounded from the martingale Hardy space $H_p(G_m)$ to the space $L_p(G_m)$ when $p>1/2.$ On the other hand, (for details see \cite{tepthesis}) boundedness  does not hold from the martingale Hardy space $H_p(G_m)$ to the space $L_p(G_m)$ when $0<p\leq1/2.$ Moreover, (see \cite{tepthesis}) there exists a martingale $f\in H_{p}(G_m),$ where $0<p<1/2$ such that

\begin{equation*}
\sup_{n\in\mathbb{N}}\left\Vert R_nf\right\Vert_{p}=\infty.
\end{equation*}

In the endpoint case $p=1/2$ it is open problem to prove divergence of the Riesz logarithmic means:

\textbf{Open Problem 3:} Does there exist a martingale $f\in H_{1/2}(G_m),$  such that
\begin{equation*}
\sup_{n\in\mathbb{N}}\left\Vert R_nf\right\Vert_{1/2}=\infty?
\end{equation*}

According to estimate \eqref{sigmanjemala} it is interesting to consider boundedness of $R_{M_n}f$  from the martingale Hardy space $H_p(G_m)$ to the space $L_p(G_m)$ when $0<p\leq 1/2,$ so we pose the following:

\textbf{Open Problem 4:} Does there exist a martingale $f\in H_{p}(G_m),$ where $0<p<1/2,$ such that
\begin{equation*}
\sup_{n\in\mathbb{N}}\left\Vert R_{M_n}f\right\Vert_{p}=\infty?
\end{equation*}

If we prove this result we immediately find that the maximal operator 
$$\sup_{n\in \mathbb{N}}\left\vert R_{M_n}f \right\vert$$
is not bounded from the martingale Hardy space $H_p(G_m)$ to the space $L_p(G_m)$ when $0<p< 1/2.$


\textbf{ Availability of data and material}

Not applicable.

\textbf{Competing interests}

The authors declare that they have no competing interests.

\textbf{Funding}

The publication charges for this manuscript is supported by the publication fund at UiT The Arctic University of Norway under code IN-1096130. 

\textbf{Authors' contributions}

All the authors contributed equally and significantly in writing this paper.
All the authors read and approved the final manuscript.

\textbf{Acknowledgements}

The work of George Tephnadze was supported by Shota Rustaveli National
Science Foundation grant FR-19-676. The publication charges for this article have been funded by a grant from the publication fund of UiT The Arctic University of Norway. 

\textbf{Author details}

$^{1}$ The University of Georgia, School of science and technology, 77a Merab Kostava St, Tbilisi 0128, Georgia.

$^{2}$ UiT The Arctic University of Norway, P.O. Box 385, N-8505, Narvik, Norway and Department of Mathematics and Computer Science, Karlstad University, 65188 Karlstad, Sweden.

$^{3}$ Department of Computer Science and Computational Engineering, UiT - The Arctic University of Norway, P.O. Box 385, N-8505, Narvik, Norway.

$^{4}$ The University of Georgia, School of science and technology, 77a Merab Kostava St, Tbilisi 0128, Georgia.

%

\end{document}